\documentclass[10pt]{article}
\usepackage{amsmath,amssymb} 
\usepackage{graphicx}
 \usepackage{authblk}
 \usepackage{a4wide}
\newcommand{\IR}{\mathbb{R}}
\newcommand{\mcT}{\mathcal{T}}
\newcommand{\mcX}{\mathcal{X}}
\newcommand{\mcK}{\mathcal{K}}
\newcommand{\mcF}{\mathcal{F}}
\newcommand{\mcKh}{\mcK_h}        
\newcommand{\mcTh}{\mcT_h}
\newcommand{\tn}{|\mspace{-1mu}|\mspace{-1mu}|}
\newcommand{\bfI}{\boldsymbol I}
\newcommand{\bfV}{\boldsymbol V}
\newtheorem{lem}{Lemma}[section]

\newtheorem{thm}{Theorem}[section]

\newenvironment{proof}{\noindent \newline {\bf Proof.}}
{\hfill \mbox{\fbox{} } \newline}

\title{\bf Cut Bogner-Fox-Schmit Elements  for Plates}
\date{}
\begin{document}

\author[$\star$]{Erik Burman}
\author[$\dagger$]{Peter Hansbo}
\author[$\ddagger$]{Mats G. Larson}
\affil[$\star$]{\footnotesize  Department of Mathematics, University College London, London, UK--WC1E  6BT, United Kingdom}

\affil[$\dagger$]{\footnotesize Department of Mechanical Engineering, J\"onk\"oping University, SE-55111 J\"onk\"oping, Sweden}

\affil[$\ddagger$]{\footnotesize Department of Mathematics and Mathematical Statistics, Ume{\aa} University, SE-90187 Ume{\aa}, Sweden}

\maketitle

\begin{abstract} We present and analyze  a method for thin plates based on cut Bogner-Fox-Schmit elements, which are $C^1$ elements obtained by taking tensor products of Hermite splines. The formulation is based on Nitsche's method for weak enforcement of essential boundary conditions together with addition of certain stabilization terms that enable us to 
establish coercivity and stability of the resulting system of linear equations. We also take geometric approximation of the boundary into account and we focus our presentation on the 
simply supported boundary conditions which is the most sensitive case for geometric approximation of the boundary.
\end{abstract}


%
\section{Introduction} The Bogner-Fox-Schmit (BFS) element \cite{BoFoSc65} is a classical $C^1$ thin plate element obtained by taking tensor products of cubic Hermite splines. The element is only $C^1$ on tensor product (rectangular) elements, which is a serious drawback since it severely limits the applicability of the resulting finite element method. However, on geometries allowing for tensor product discretization it is generally considered to be one of the most efficient elements for plate analysis, cf. \cite[p. 153]{ZiTa00}.
It is also a reasonably low order element for plates which is very simple to implement, in contrast with triangular elements which either use higher order polynomials, such as the Argyris element \cite{ArFrSc68}, or macro element techniques, such as the Clough--Tocher element \cite{ClTo65}. The construction of curved versions  of these elements for boundary fitting can also be cumbersome, see, e.g., \cite{BeBo93}. It should be noted that the use of straight line segments for discretizing the boundary is not to be recommended, not only because of accuracy issues but also due to Babu\v ska's paradox for simply supported plates, see \cite{BaPi90}. 

To remedy the problem of geometry discretization for the BFS element,
we herein develop a cut finite element version, allowing for discretizing a smooth boundary 
which may cut through the tensor product mesh in an arbitrary manner. Adding stabilization 
terms on the faces associated with elements that intersect the boundary, we obtain a stable method with optimal order convergence. We prove a priori error estimates which also take 
approximation of the boundary into account. The focus of the analysis is on simply supported boundary conditions, the computationally most challenging case.  

The paper is organized as follows. In Section 2, we recall the thin plate Kirchhoff model;
in Section 3 we formulate the cut finite element method; in Section 4 we present the analysis of the method starting with a sequence of technical results 
leading up to a Strang Lemma and an estimate of the consistency error and finally a priori error estimates in the energy and $L^2$ norms. In Section 5, we present some numerical illustrations, and in Section 6 some concluding remarks.

\section{The Kirchhoff Plate} 

Consider a simply supported thin plate in a domain $\Omega\subset \IR^2$ with smooth 
boundary $\partial \Omega$. The displacement $u:\Omega \rightarrow \IR$ satisfies
\begin{equation}\label{eq:plate-strong}
\nabla \cdot (\sigma(\nabla u) \cdot \nabla) = f
\end{equation}
where the stress tensor is 
\begin{equation}\label{eq:stress}
\sigma(\nabla v) = \kappa (\epsilon(\nabla v) +\nu (1-\nu)^{-1} (\nabla \cdot (\nabla v) ) I
= \kappa( \nabla \otimes \nabla v + \nu (1-\nu)^{-1} (\Delta v) I))
\end{equation}
where the strain tensor is defined by
\begin{equation}\label{eq:plate-strain}
\epsilon (\nabla v) = ( (\nabla v) \otimes \nabla  + \nabla \otimes (\nabla v) )/2 = \nabla \otimes \nabla v
\end{equation}
and $\kappa$ is the parameter
\begin{equation}
\kappa =  \frac{E t^3}{12(1+\nu)}  
\end{equation}
with $E$ the Young's modulus, $\nu$ the Poisson's ratio, and $t$ the 
plate thickness. Since $0 \leq  \nu \leq 0.5$ both $\kappa$ and 
$\nu(1-\nu)^{-1}$ are uniformly bounded. 

We shall focus our presentation on simply supported boundary conditions
\begin{equation}\label{eq:plate-ss}
u = 0 \quad \text{on $\partial \Omega$}, \qquad M_{nn}(u) = 0 \quad \text{on $\partial \Omega$}
\end{equation}
where the moment tensor $M$ is defined by
\begin{equation}\label{eq:moment}
M(u) = \sigma(\nabla u)
\end{equation}
and $M_{ab} = a\cdot M \cdot b$ for $a,b \in \IR^2$.

The weak form of (\ref{eq:plate-strong}) and (\ref{eq:plate-ss}) takes the form: find 
$u \in V = \{ v \in H^2(\Omega) : \text{$v=0$ on $\partial \Omega$}\}$ such that 
\begin{equation}\label{eq:weak}
a(u,v) = l(v) \qquad \forall v \in V
\end{equation}
where 
\begin{equation}\label{eq:form-a}
a(v,w) = (\sigma(\nabla v),\epsilon(\nabla w))_\Omega 
= 
\kappa ( (\nabla \otimes v,\nabla \otimes \nabla w)_\Omega + \nu (1-\nu)^{-1} (\Delta v,\Delta w)_\Omega)
\end{equation}
and $l(v) = (f,v)_\Omega$. The form  $a$ is symmetric, continuous, and coercive on 
$V$ equipped with the $H^2(\Omega)$ norm and it follows from the Lax-Milgram 
theorem that there exists a unique solution in $V$ to the (\ref{eq:weak}). Furthermore, 
for smooth boundary and $f \in L^2$ we have the elliptic regularity 
\begin{equation}\label{eq:elliptic-reg}
\| u \|_{H^4(\Omega)} \lesssim \| f \|_\Omega
\end{equation}

\section{The Finite Element Method}
\paragraph{The Mesh and Finite element Space.}
We begin by introducing the following notation.
\begin{itemize}
\item Let $\widetilde{\mcT}_h$, $h\in (0,h_0]$, be a family of 
partitions of $\IR^2$ into squares with side $h$.  Let $\widetilde{V}_h$ be the 
Bogner-Fox-Schmit space consisting of tensor products of cubic Hermite 
splines on $\widetilde{\mcT}_h$.

\item Let $\rho$ be the signed distance function associated with $\partial \Omega$ 
and let $U_\delta(\partial \Omega) = \{x\in \IR^2 : |\rho(x)|< \delta\}$ be the tubular neighborhood 
of $\partial \Omega$ of thickness $2\delta$. Then there is $\delta_0>0$ such that the closest point mapping $p:U_{\delta_0}(\partial \Omega) \rightarrow \partial \Omega$ is a well defined function of the form $p(x) = x - \rho(x) n(p(x))$. 

\item Let $\{\Omega_h, h \in (0,h_0]\}$ be a family of approximations of 
$\Omega$ such that $\partial \Omega_h \subset U_{\delta_0}(\partial \Omega)$ 
is piecewise smooth  and
\begin{equation}\label{eq:omegah-dist}
\| \rho \|_{L^\infty(\partial \Omega_h)} \lesssim h^4
\end{equation}
\begin{equation}\label{eq:omegah-normal}
\|n(p) - n_h\|_{L^\infty(\partial \Omega_h)} \lesssim h^3
\end{equation}
Furthermore, we assume that for each element $T$ such that $\partial \Omega_h$ intersects 
the interior of $T$, i.e. $\text{int}(T) \cap \partial \Omega_h \neq \emptyset$, the curve segment $\partial \Omega_h \cap T$ is smooth and intersect the boundary $\partial T$ of $T$ in precisely  two different points. Let $\mcX_h$ be the set of all points where $\partial 
\Omega_h$ is not smooth.

\item 
Let $\mcT_h = \{T \in \widetilde{\mcT}_h : T \cap \Omega_h \neq \emptyset\}$ 
be the active mesh. Let $\mcT_{h,I}$ be the set of elements such that $T\subset \Omega$ 
and let $\mcF_{h,I}$ be the set of interior faces in $\mcT_{h,I}$. Let $\mcT_{h,B} = \mcT_h \setminus \mcT_{h,I}$ and $\mcF_{h,B} = \mcF_h \setminus \mcF_{h,I}$.

\item Let $V_h$ be the restriction of $\widetilde{V}_h$ to  $\mcT_h$. Let $\mcKh = \mcTh \cap \Omega_h$ be the intersection of the active elements $T$ with $\Omega_h$.
\end{itemize}

\paragraph{The Finite Element Method.}
The method reads: find $u_h \in V_h$ such that 
\begin{equation}\label{eq:fem}
\boxed{
A_h(u_h,v) = l_h(v) \qquad \forall v \in V_h
}
\end{equation}
The forms are defined by
\begin{align}\label{eq:Ah}
A_h(v,w) &= a_h(v,w)+\beta s_h(v,w) 
\\ \label{eq:ah}
a_h(v,w) &=  (\sigma(\nabla v), \epsilon(\nabla w ) )_{\Omega_h} 
+(T(v),w)_{\partial \Omega_h} +(v,T(w))_{\partial \Omega_h} + \gamma h^{-3} (v,w)_{\partial \Omega}
\\ \label{eq:sh}
s_h(v,w) &= h ( [ \nabla^2_n v],  [ \nabla^2_n w])_{\mcF_{h,B}}  
+ h^3 ( [ \nabla^3_n v],  [ \nabla^3_n w])_{\mcF_{h,B}} 
\\
l_h(v) &= (f, v)_{\Omega_h}
\end{align}
where 
\begin{equation}
T = (M\cdot \nabla)_n + \nabla_t M_{nt}
\end{equation}
with sub-indices $n$ and $t$ indicating scalar product with the normal and tangent of 
$\partial \Omega_h$, and $\beta, \gamma$ are positive parameters which are proportional to $\kappa$. Here $s_h$ is a stabilization form, which provides the necessary control of 
the cut elements, see (\ref{eq:stab-pairwise}). The bilinear form, apart from the stabilization terms, stems from Nitsche's method \cite{Nit70}, first analyzed for plates 
in a discontinuous Galerkin setting in \cite{HaLa02}.

\section{Error Estimates}

\subsection{Basic Properties of $\boldmath{A}_{\boldmath{h}}$}
\paragraph{The Energy Norm.}
Define the following energy norm on $V+V_h$, 
\begin{align}\label{eq:energy-norm}
\tn v \tn_h^2&=\tn v \tn^2_{\Omega_h} +   \beta \| v \|^2_{s_h} 
+ h^3 \| T(v ) \|^2_{ \partial \Omega_h} + h^{-3}\| v \|^2_{\partial \Omega} 
\end{align}
where 
\begin{equation}
\tn v \tn^2_{\Omega_h} =  ( \sigma(\nabla v),\epsilon(\nabla v)_{\Omega_h} 
\end{equation}
and we use the standard notation $\| v \|^2_{s_h} = s_h(v,v)$. 

\paragraph{Stabilization.}
The stabilization term provides us with the following bound
\begin{align}\label{eq:stab-global}
\boxed{
\| \nabla^j v \|^2_{\mcTh} \lesssim \| \nabla^j v \|_{\mcT_{h,I}} + h^{2(2-j)} \| v \|^2_{s_h}, 
\qquad j=0,1,2,3
}
\end{align}
which follows from the standard estimate 
\begin{equation}\label{eq:stab-pairwise}
\| \nabla^j v \|^2_{T_2 }\lesssim  \| \nabla^k j \|^2_{T_1 } 
+ \sum_{k=j}^p h^{2(k-j)} \| [ \nabla^k v ] \|^2_{F}
\end{equation}
where $T_1$ and $T_2$ are elements that share the edge $F$, and 
$v|_{T_i} \in \mathbb{P}_p(T_i)$, the space of polynomials of order $p$. See 
for instance \cite{HaLaLa17}, \cite{MasLarLog14} for further details.

\paragraph{Continuity and Coercivity.}
The form $A_h$ is continuous
\begin{align}\label{eq:continuity }
A_h(v,w) \lesssim \tn v \tn_h \tn w \tn_h \qquad v,w \in V + V_h
\end{align}
which follows directly from the Cauchy-Schwarz inequality, and for $\gamma$ 
large enough coercive
\begin{align}\label{eq:coercivity}
\tn v \tn_h^2 \lesssim A_h(v,v) \qquad v \in V_h
\end{align}
{\noindent\newline {\bf Verification of (\ref{eq:coercivity}).}}
Using inverse inequalities followed by the stabilization estimate (\ref{eq:stab-global}) we 
obtain
\begin{align}
\kappa^{-1} h^3 \| T(v) \|^2_{\partial \Omega_h} 
&\lesssim \kappa h^2 \| \nabla^3 v \|^2_{\mcTh (\partial \Omega_h)} 
\lesssim \kappa \| \nabla^2 v \|^2_{\mcTh (\partial \Omega_h)} 
\\
&\qquad 
\lesssim  \kappa (\| \nabla^2 v \|^2_{\Omega} + \| v \|^2_{s_h})
\lesssim \tn v \tn^2_{\Omega_h} + \kappa \| v \|^2_{s_h}
\end{align}
and thus there is a constant $C_*$ such that 
\begin{align}
\kappa^{-1} h^3 \| T(v) \|^2_{\partial \Omega_h}  \leq C_* ( \tn v \tn^2_{\Omega_h} + \kappa \| v \|^2_{s_h} )
\end{align}
We then have 
\begin{align}
A_h(v,v) &= \tn v \tn^2_{\Omega_h} + \beta \| v \|^2_{s_h} 
- 2(T(v),v)_{\partial \Omega_h} + \gamma h^{-3} \| v \|^2_{\partial \Omega_h}
\\
&\geq 
\tn v \tn^2_{\Omega_h} + \beta \| v \|^2_{s_h} 
-  \delta \kappa^{-1}h^3 \|T(v)\|^2_{\partial \Omega_h}  + (\gamma - \delta^{-1} \kappa) h^{-3} \|v\|^2_{\partial \Omega_h} 
\\
&\geq 
(1 - C_* \delta ) \tn v \tn^2_{\Omega_h} + (\beta - \kappa C_* \delta) \| v \|^2_{s_h} 
+ (\gamma - \delta^{-1} \kappa )h^{-3} \|v\|^2_{\partial \Omega_h} 
\end{align}
and we find that taking $\delta$ small enough to guarantee that $1 - C_* \delta \geq m>0$, 
$\beta$ large enough to guarantee that $\beta - \kappa C_* \delta \geq m$, and $\gamma$ large enough  to guarantee that  $\gamma - \delta^{-1} \kappa h^{-3} \geq m$ 
the coercivity (\ref{eq:coercivity}) follows.
{\hfill \mbox{\fbox{} } \newline}
%
%
%
%

\subsection{Interpolation}

Let $I_{h}:C^1(\IR^2) \rightarrow V_h$ be the standard element wise interpolant 
associated with the degrees of freedom in $V_h$. Then we have the estimate
\begin{equation}\label{eq:interpol-element}
\| v - I_{h} v \|_{H^m(T)} \lesssim h^{4-m} \| v \|_{H^4(T)}\qquad m=0,1,2,3
\end{equation}
To construct an interpolation operator for cut elements we recall that given $v \in H^s(\Omega)$ there is an extension operator $E:H^s(\Omega) \rightarrow H^s(\IR^2)$ such that 
\begin{align}\label{eq:ext-stability}
\|E v \|_{H^s(\IR^2)} \lesssim \| v \|_{H^s(\Omega)} 
\end{align}
for all $s>0$. Then we define the interpolation operator
\begin{equation}\label{eq:def-interpol-cut}
C^1(\Omega)  \ni v \mapsto I_{h} (E v ) = \pi_h v \in V_{h}
\end{equation}
Combining (\ref{eq:interpol-element}) with (\ref{eq:ext-stability}) we obtain the 
interpolation error estimate
\begin{equation}\label{eq:interpol-general}
\boxed{\| v - \pi_h v \|_{H^m(\mcTh)} \lesssim  h^{4-m} \| v \|_{H^4(\Omega)}\qquad m=0,1,2,3}
\end{equation}
For the energy norm we have the estimate
\begin{equation}\label{eq:interpol-energy}
\boxed{\tn v - \pi_h v \tn_h \lesssim h^2 \| v \|_{H^4(\Omega)}}
\end{equation}
{\noindent \newline {\bf Verification of (\ref{eq:interpol-energy}).}} Let $\eta =  v - \pi_h v$ 
and recall that
\begin{align}
\tn \eta \tn^2_h = \tn \eta \tn^2_{\Omega_h} + \| \eta \|^2_{s_h}
+ h^3\| T(\eta) \|^2_{\partial \Omega_h}
+ h^{-3} \| \eta \|^2_{\partial \Omega_h} 
\end{align}
The first term is directly estimated using (\ref{eq:interpol-general}), 
\begin{equation}
 \| \eta \|^2_{a_h} \lesssim h^4 \| v \|^2_{H^4(\Omega)}
\end{equation}
For the second term we employ the trace inequality 
\begin{align}\label{eq:trace-element}
\| w \|^2_{\partial T} \lesssim h^{-1} \| w \|^2_T + h \| \nabla w \|^2_T
\end{align} 
 to conclude that 
\begin{align}
&\| \eta \|^2_{s_h} = \sum_{j=2}^3 h^{2j - 3} \| [ \nabla_n^j \eta ] \|^2_{\mcF_{h,B}}
\lesssim  \sum_{j=2}^3 h^{2j - 3} 
(h^{-1} \| \nabla_n^j \eta \|^2_{\mcT_{h,B}} + h \| \nabla_n^{j+1} \eta \|^2_{\mcT_{h,B}} )
\\
&\qquad \lesssim  \sum_{j=2}^3 h^{2j - 4} 
(\| \nabla_n^j \eta \|^2_{\mcT_{h,B}} + h^2 \| \nabla_n^{j+1} \eta \|^2_{\mcT_{h,B}} )
\lesssim h^4 \| v \|^2_{H^4(\mcT_{h,B})}
\lesssim h^4 \| v \|^2_{H^4(\Omega)}
\end{align}
For the third term we use the trace inequality 
\begin{align}\label{eq:trace-omegah}
\| v \|^2_{\partial \Omega_h} \lesssim \delta^{-1} \| v \|^2_{U_\delta(\partial \Omega)\cap \Omega} + \delta \| \nabla v \|^2_{U_\delta(\partial \Omega)\cap \Omega} 
\end{align}
with $\delta\sim h$,
\begin{align}
h^3 \|T(\eta)\|^2_{\partial \Omega_h} 
&\lesssim 
h^3 ( \delta^{-1} \| \nabla^3 \eta \|^2_{\mcT_h( U_\delta(\partial \Omega)\cap \Omega)} 
+ \delta \| \nabla^4 \eta \|^2_{\mcTh(U_\delta(\partial \Omega)\cap \Omega)} )
\lesssim h^4 \| v \|^2_{H^4(\Omega)}
\end{align}
Finally, the fourth term is estimated in the same way as the third,
\begin{align}
h^{-3}\| \eta  \|^2_{\partial \Omega_h} 
&\lesssim
h^{-3} ( \delta^{-1} \| \eta \|^2_{\mcTh(U_\delta(\partial \Omega) \cap \Omega)} 
+ \delta\| \nabla \eta  \|^2_{\mcTh(U_\delta(\partial \Omega) \cap \Omega)}  )
\lesssim
h^4 \| v \|^2_{H^4(\mcTh(\partial \Omega_h))}
\end{align}
which completes the verification of (\ref{eq:interpol-energy}). 
{\hfill \mbox{\fbox{} } \newline}

\subsection{Consistency Error Estimate}

 \begin{lem} Let $u$ be the exact solution to (\ref{eq:plate-strong}) with boundary 
 conditions  (\ref{eq:plate-ss}),  and $u_h$ the finite element approximation defined 
 by (\ref{eq:fem}),  then 
 \begin{equation}\label{eq:strang}
\boxed{
\tn u - u_h \tn_{h} \lesssim \tn u - \pi_h u \tn_h + \sup_{v \in V_h\setminus \{0\}}
\frac{A_h(u,v) - l_h(v)}{\tn v \tn_h}
}
\end{equation}
 \end{lem}
 \begin{proof}
Adding and subtracting an interpolant we obtain 
\begin{equation}
\tn u - u_h \tn_{h} 
\leq \tn u - \pi_h u \tn_h + \tn \pi_h u - u_h \tn_h 
\end{equation}
Using coercivity we can estimate the second term on the right hand side as follows
\begin{align}
\tn \pi_h u - u_h \tn_{h} &\leq 
\sup_{v \in V_h\setminus \{0\}}
\frac{A_h(\pi u - u_h,v) }{\tn v \tn_h}
\\
&\leq 
\sup_{v \in V_h\setminus \{0\}} \frac{A_h(\pi u - u,v) }{\tn v \tn_h} 
+ \sup_{v \in V_h\setminus \{0\}} \frac{A_h(\pi u - u_h,v) }{\tn v \tn_h}
\\
&\leq \tn \pi_h u - u \tn_h + \sup_{v \in V_h\setminus \{0\}} \frac{A_h(\pi u,v) - l_h(v)}{\tn v \tn_h}
\end{align}
where we added and subtracted $u$ in the numerator and for the first term used the estimate $A_h(\pi u - u,v)  \lesssim \tn \pi_h u - u_h \tn_h \tn v \tn_h$ and for the 
second used (\ref{eq:fem}) to eliminate $u_h$. Combining the estimates
 the desired result follows directly.
 \end{proof}

\begin{lem} Let $\varphi \in H^4(\IR^2)$ and $v \in V+V_h$, then
\begin{align}\label{eq:part-int}
(\nabla \cdot (M(\varphi) \cdot \nabla),v )_{\Omega_h} 
&= (M(\varphi) , \epsilon(\nabla v) )_{\Omega_h} 
- ( M_{nn}(\varphi), \nabla_n v)_{\partial \Omega_h} 
\\
&\qquad 
+ (T(\varphi),v)_{\partial \Omega_h} 
+ ([M_{nt}],v)_{\mcX_h} 
\end{align}
where, for $x\in \mcX_h$, $[M_{nt}]_x$ is defined by
\begin{align}
[M_{nt}] |_x = M(x)_{n_h^+ t_h^+} - M(x)_{n_h^- t_h^-} 
\end{align}
In the case of $C^1$ boundary  $(v, [M_{nt}])_{\mcX_h}=0$.
\end{lem}

\begin{proof} Using the simplified notation $M = M(\varphi)$ and $T=T(\varphi)$ 
for brevity we obtain by integrating by parts 
\begin{align}
(\nabla \cdot (M\cdot \nabla),v)_{\Omega_h} 
&=( (M \cdot \nabla)_n,v )_{\partial \Omega_h}
- (M \cdot \nabla, \nabla v )_{\Omega_h} 
\\  \label{eq:part-int-a}
&=  ((M \cdot \nabla)_n, v )_{\partial \Omega_h}
- (M_n, \nabla v )_{\partial \Omega_h}
+ (M, \epsilon(\nabla v))_{\Omega_h} 
\end{align}
Splitting $\nabla v$ in tangent and 
normal contributions on $\partial \Omega_h$, we have the identity
\begin{align}
 (\nabla v, M_n )_{\partial \Omega_h \cap T}
&=  
(\nabla_n v, M_{nn} )_{\partial \Omega_h \cap T}
 +  (\nabla_t v, M_{nt} )_{\partial \Omega_h \cap T} 
 \\ 
& =  
(\nabla_n v, M_{nn} )_{\partial \Omega_h\cap T} -  (v, \nabla_t M_{nt} )_{\partial \Omega_h \cap T} + (v, M_{nt} t \cdot \nu)_{ \partial (\partial \Omega_h \cap T )}
\end{align}
where we integrated by parts along the curve segments $\partial \Omega_h \cap T$, and 
$\nu$ is the exterior unit tangent vector to $\partial \Omega_h \cap T$. Summing over all elements that intersect $\partial \Omega_h$, we obtain the identity 
\begin{align}\label{eq:part-int-b}
 (\nabla v, M_n )_{\partial \Omega_h}
& =  
(\nabla_n v, M_{nn} )_{\partial \Omega_h} -  (v, \nabla_t M_{nt} )_{\partial \Omega_h} + (v, [M_{nt}])_{\mcX_h}
\end{align}
Combining (\ref{eq:part-int-a}) and (\ref{eq:part-int-b}), we obtain
\begin{align}
((v,\nabla \cdot (M\cdot \nabla) )_{\Omega_h} 
&=  (\epsilon(\nabla v) ,M)_{\Omega_h}  
-  (\nabla_n v, M_{nn} )_{\partial \Omega_h}
\\
&\qquad 
+ (v, (M \cdot \nabla)_n  +\nabla_t M_{nt}  )_{\partial \Omega_h}
- (v, [M_{nt}])_{\mcX_h} 
\end{align}
and setting $T = (M \cdot \nabla)_n  +\nabla_t M_{nt}$ we obtain the desired result.
\end{proof}

\begin{lem} Let $u$ be the exact solution to (\ref{eq:plate-strong}) with boundary conditions 
 (\ref{eq:plate-ss}), then there is a constant such that for all $v\in V_h$,
\begin{align} \label{eq:consistency-bound}
\boxed{A_h(u ,v) - l_h(v) \lesssim  
h^4 \| u \|_{H^4(\Omega)} \tn v \tn_{h,\bigstar}
\lesssim 
h^{5/2} \| u \|_{H^4(\Omega)}  \tn v \tn_{h}}
\end{align}
where $\tn v \tn_{h,\bigstar}$ is the norm
\begin{equation}\label{eq:energ-star-norm}
\tn v \tn^2_{h,\bigstar} = \tn v \tn_h^2 + \| T(v) \|^2_{\partial \Omega_h} + h^{-6} \| v \|^2_{\partial \Omega_h}  \leq (1+ h^{-3}) \tn v \tn^2_h
\end{equation}
\end{lem}

\begin{proof}  Using the definition (\ref{eq:Ah}), the fact that $s_h(u,v) = 0$ for 
$u \in H^4(\Omega)$, and the partial integration identity (\ref{eq:part-int})  we obtain
\begin{align}
A_h(u,v)- l_h(v)
&= (M(u),\epsilon(\nabla v))_{\Omega_h} 
+ (T(u),v)_{\partial \Omega_h} + (u,T(v))_{\partial \Omega_h}
\\
&\qquad 
+ \gamma h^{-3}(u,v)_{\partial \Omega_h} 
- (\nabla \cdot ( M(u) \cdot \nabla ) ,v)_{\Omega_h} 
\\
&=
(M_{nn}(u), \nabla_n v )_{\partial \Omega_h} 
+ ([M_{nt}],v)_{\mcX_h}
\\ 
&\qquad
+(u, T(v))_{\partial \Omega_h} 
+ \gamma h^{-3} (u,v)_{\partial \Omega_h} 
\\
&=I + II + III + IV
\end{align}

Before turning to the estimates of $I-IV$ we first recall the following estimates from \cite{BuHaLa18}. 
There is a constant such that for all 
$w\in H^1_0(\Omega)$,
\begin{align}\label{eq:wbnd-bound-weak}
\|w \|_{\partial \Omega_h} 
\lesssim \delta^{1/2} \| w \|_{H^1(U_\delta(\partial \Omega))} 
\lesssim \delta^{1/2} \| w \|_{H^1(U_{\delta_0}(\partial \Omega)\cup \Omega)} 
\lesssim  \delta^{1/2} \| w \|_{H^1(\Omega)} 
\end{align}
for $0< \delta < \delta_0$ such that $\partial \Omega_h \subset U_\delta(\partial \Omega_h)$, where we at last used the stability (\ref{eq:ext-stability}) of the extension. 
In view of the geometry approximation assumption  (\ref{eq:omegah-dist}) we may take $\delta \sim h^4$. Furthermore,we may strengthen the estimate as follows 
\begin{align}\label{eq:wbnd-bound-strong}
\|w \|_{\partial \Omega_h} 
\lesssim \delta \| w \|_{W_1^\infty(U_\delta(\partial \Omega))} 
\lesssim \delta \| w \|_{W_1^\infty(U_\delta \cup \Omega)}
\lesssim \delta \| w \|_{H^{2+\epsilon}(U_{\delta_0} \cup \Omega)}
\lesssim \delta \| w \|_{H^{2+\epsilon}(\Omega)}
\end{align}
where we used the Sobolev embedding theorem and the stability (\ref{eq:ext-stability})
of the extension operator and we may take $\delta \sim h^4$.

\paragraph{$\bfI$.} Using (\ref{eq:wbnd-bound-strong}) with $w = M_{nn}(u)$,
\begin{align}
&(M_{nn}(u), \nabla_n v )_{\partial \Omega_h} 
\lesssim 
\| M_{nn}(u)\|_{\partial \Omega_h} \|\nabla_n v \|_{\partial \Omega_h} 
\lesssim 
\delta \| u \|_{H^4(\Omega)} \tn v \tn_h
\end{align}
where we finally used the estimate 
\begin{equation}\label{eq:traceomega_h}
\|\nabla_n v \|_{\partial \Omega_h} \lesssim  \tn v \tn_h
\end{equation}
which we prove as follows. Recalling that $\partial \Omega_h \subset U_{\delta_0}$ for 
all $h \in (0,h_0]$ it follows that $\Omega \setminus U_{\delta_0}(\partial \Omega) 
\subset \Omega_h$ for all $h \in (0,h_0]$. Since $\Omega \setminus U_{\delta_0}(\partial \Omega)$ is independent of $h$ we have the trace inequality 
\begin{equation}
\|\nabla v \|_{\partial (\Omega \setminus U_{\delta_0}(\partial \Omega) )} 
\lesssim 
\| v \|_{H^2(\Omega \setminus U_{\delta_0}(\partial \Omega) )} 
\end{equation}
with hidden constant independent of $h$. We then obtain
\begin{align}
\|\nabla_n v \|^2_{\partial \Omega_h}  
&\lesssim \|\nabla v \|^2_{\partial \Omega_h}  
\lesssim 
\delta_0 \|\nabla^2 v \|^2_{\Omega_h \setminus (\Omega \setminus U_{\delta_0}(\partial \Omega) )} + \|\nabla v \|^2_{\partial (\Omega \setminus U_{\delta_0}(\partial \Omega) ) }  
\\
&\qquad 
\lesssim
\delta_0 \|\nabla^2 v \|^2_{\Omega_h \setminus (\Omega \setminus U_{\delta_0}(\partial \Omega) )} 
+ \| v \|_{H^2(\Omega \setminus U_{\delta_0}(\partial \Omega) )} 
\lesssim  \tn v \tn_h^2
\end{align}

\paragraph{$\bfI\bfI$.} Using the assumption on the accuracy of the discrete normal 
(\ref{eq:omegah-normal}) we have for each $x\in \mcX_h$,
\begin{align}
|[M_{nt}]| = M^+_{n_h t_h} - M^-_{n_h t_h}=  M^+_{n_h t_h} -M_{n t}
+ M_{n t}  -  M^-_{n_h t_h}
\end{align}
where the first term on the right hand side can be estimated as follows 
\begin{equation}
|M_{n_h t_h}  -  M^-_{n t} |\leq |(n_h -n )\cdot M \cdot t_h |+ | n\cdot M \cdot (t_h - t)|
\lesssim h^3 |M|
\end{equation}
We then have
\begin{align}
([M_{nt}],v)_{\mcX_h} 
&\leq \|[M_{nt}]\|_{\mcX_h} \|v\|_{\mcX_h}
\lesssim h^3 \| M \|_{\mcX_h} \|v\|_{\mcX_h}
\lesssim h^2 h^{1/2}\| M \|_{\mcX_h} h^{1/2} \|v\|_{\mcX_h}
\\
&\qquad 
\lesssim h^2 \| M \|_{L^\infty(\mcX_h)} h^2 \tn v \tn_h
\lesssim h^4 \| u \|_{H^4(\Omega)} \tn v \tn_h
\end{align}
where we used the fact that the number of elements , denoted by
$| \mcX_h |$, in $\mcX_h$  satisfies $| \mcX_h | \sim h^{-1}$, and the Sobolev inequality 
to obtain $h \| M \|^2_{\mcX_h} \lesssim \| u \|^2_{H^4(\Omega)}$, and the estimate
\begin{align}
h \|v\|^2_{\mcX_h}
&\lesssim
\| v \|^2_{\partial \Omega_h} 
\lesssim 
h^3 \tn v \tn^2_h
\end{align}
Here the second estimate follows directly from the definition of the energy norm and to 
verify the first consider $x\in \mcX_h$ and let $B_r(x)$ be a ball of radius $r \sim h$ 
centred at $x$. Let $T_x\in \mcT_h$ be one of the elements such that $x \in \partial T_x$ 
and given $v\in V_h$ let $v_x$ be the extension to $\IR^2$ of $v|_{T_x}$. 
We then have 
\begin{equation}\label{eq:consist-II-d}
h \| v \|^2_{\mcX_h} 
\lesssim 
\sum_{x \in \mcX_h} h |v(x)|^2
\lesssim 
\sum_{x \in \mcX_h} \| v_x \|^2_{\partial \Omega_h \cap B_r(x)}
\end{equation}
where we used the fact that $|{\partial \Omega_h \cap B_r(x)}| \sim h$, which follows 
from (\ref{eq:omegah-dist}) and (\ref{eq:omegah-normal}) together with a change of coordinates to the exact surface. Let $\mcT_{h,x} = \mcT_h(B_r(x))$, let $\mcF_{h,x}$ be the interior faces in $\mcT_{h,x}$, and let $s_{h,x}$ be defined by (\ref{eq:sh}) with $\mcF_{h,B}$ 
replaced by $\mcF_{h,x}$. We then have the bound
\begin{align}
\| v - v_x \|^2_{\mcT_{h,x}} 
\lesssim 
h^4 \| v \|^2_{s_{h,x}}
\end{align}
which is a local version of (\ref{eq:stab-global}) on the patch $\mcT_{h,x}$. Adding and subtracting $v$ we have 
\begin{align}
\| v_x \|^2_{\partial \Omega_h \cap B_r(x)} &\lesssim 
\| v - v_x \|^2_{\partial \Omega_h \cap B_r(x)} + \| v \|^2_{\partial \Omega_h \cap B_r(x)}
\\
&\lesssim 
h^{-1} \| v - v_x \|^2_{\mcT_{h,x}} + \| v\|^2_{\partial \Omega_h \cap B_r(x)}
\\ \label{eq:consist-II-e}
&\lesssim 
h^3 \| v \|^2_{s_{h,x}} + \| v\|^2_{\partial \Omega_h \cap B_r(x)}
\end{align}
Combining (\ref{eq:consist-II-d}) and (\ref{eq:consist-II-e}) we obtain
\begin{align}
\| v \|^2_{\mcX_h} \lesssim \sum_{x \in \mcX_h} h^3 \| v \|^2_{s_{h,x}} + \| v\|^2_{\partial \Omega_h \cap B_r(x)}
\lesssim h^3 \tn v \tn_h
\end{align}
where we used the fact that the number of balls $B_r(y)$, $y\in \mcX_h$, that intersect $B_r(x)$ is uniformly bounded independent of $x \in \mcX_h$ and $h \in (0,h_0]$.

\paragraph{$\bfI\bfI\bfI$.} Using (\ref{eq:wbnd-bound-strong}) with $w=u$,
\begin{align}
(u, T(v))_{\partial \Omega_h} &\leq \|u\|_{\partial \Omega_h}  \|T(v)\|_{\partial \Omega_h} 
\lesssim \delta \| u \|_{H^4(\Omega)} \|T(v)\|_{\partial \Omega_h}  
\\
&\qquad \lesssim \delta h^{-3/2} \| u \|_{H^4(\Omega)} h^{3/2} \|T(v)\|_{\partial \Omega_h}  
\lesssim  h^{5/2} \| u \|_{H^4(\Omega)}\tn v \tn_h
\end{align}

\paragraph{$\bfI\bfV$.} Proceeding in the same way as for Term $III$,
\begin{align}
h^{-3}(u,v)_{\partial \Omega_h} 
&\lesssim h^{-3} \| u \|_{\partial \Omega_h} \| v \|_{\partial \Omega_h} 
\lesssim \delta h^{-3} \| u \|_{H^4(\Omega)} \| v \|_{\partial \Omega_h} 
\\
&\qquad \lesssim  \delta h^{-3/2} \| u \|_{H^4(\Omega)} \tn v \tn_h 
\lesssim  h^{5/2}  \| u \|_{H^4(\Omega)} \tn v \tn_h
\end{align}
Combining the estimates we find that 
\begin{align}
A_h(u,v) - l_h(v) &\lesssim  h^4 \| u \|_{H^4(\Omega)} \tn v \tn_h  
+ h^4  (\|T(v) \|_{\partial \Omega_h} + h^{-3} \|v \|_{\partial \Omega_h} ) 
\\
&\lesssim 
h^{5/2} \| u \|_{H^4(\Omega)} \tn v \tn_h  
\end{align}
which completes the proof.
%

\end{proof}

\subsection{Error Estimates}

\begin{thm} The finite element solution defined by (\ref{eq:fem}) satisfies
\begin{equation}\label{eq:error-est-energy}
\boxed{\tn u - u_h \tn_h \lesssim h^2 \| u \|_{H^4(\Omega)}}
\end{equation}
\end{thm}
\begin{proof}
 Using the second bound of (\ref{eq:consistency-bound}) in 
(\ref{eq:strang}) followed by the interpolation error bound (\ref{eq:interpol-energy}) 
we directly get the desired estimate.
\end{proof}


\begin{thm}   The finite element solution defined by (\ref{eq:fem}) satisfies
\begin{equation}\label{eq:error-est-Ltwo}
\boxed{\| u - u_h \|_{\Omega_h} \lesssim h^4 \| u \|_{H^4(\Omega)}}
\end{equation}
\end{thm}

\begin{proof} Adding and subtracting an interpolant and using the interpolation error estimate (\ref{eq:interpol-general})  we have the  estimate
\begin{align}
\| u - u_h \|_{\Omega_h} &\leq \| u - \pi_h u \|_{\Omega_h} 
+ \| \pi_h u - u_h \|_{\Omega_h}
\\
&\lesssim 
h^4 \| u \|_{H^4(\Omega)} +  \| \pi_h u - u_h \|_{\Omega_h} 
\end{align}
In order to estimate $\| \pi_h u - u_h \|_{\Omega_h}$ we let $\phi_h\in V_h$ be the 
solution to the discrete dual problem 
\begin{align}
(v,\psi)_{\Omega_h} &= A_h(v,\phi_h) \qquad \forall v \in V_h
\end{align}
Setting $v = \pi_h u - u_h$ we obtain the error  representation 
\begin{align}
(\pi_h u - u_h,\psi)_{\Omega_h} &= A_h( \pi_h u - u_h, \phi_h)
\\
&= A_h( \pi_h u - u, \phi_h) + A_h( u - u_h, \phi_h)
\\
&= \underbrace{A_h( \pi_h u - u, \phi_h - \phi)}_{I} + 
\underbrace{A_h( \pi_h u - u,\phi)}_{II} 
+ \underbrace{A_h( u,\phi_h) - l_h(\phi_h)}_{III}
\end{align}
where $\phi$ is the solution to the continuous dual problem 
\begin{equation}
\label{eq:plate-strong-dual}
\nabla \cdot (\sigma(\nabla \phi) \cdot \nabla) = \psi \quad \text{in $\Omega$},
\qquad
\phi = M_{nn}(\phi) = 0  \quad \text{in $\Omega$}
\end{equation}
\paragraph{$\bfI$.} Since $\phi_h$ is the finite element approximation of $\phi$ we have 
the error estimate
\begin{align}
\tn \phi - \phi_h \tn_h \lesssim h^2 \| \phi \|_{H^4(\Omega)} \lesssim h^2 \| \psi \|_{\Omega} 
\end{align}
where we used elliptic regularity (\ref{eq:elliptic-reg}), which directly gives
\begin{align}
A_h( \pi_h u - u, \phi_h - \phi) 
&\leq 
\tn \pi_h u - u \tn_h \tn \phi_h - \phi \tn_h 
\lesssim
h^4 \| u \|_{H^4(\Omega)} \| \psi\|_{\Omega} 
\end{align}

\paragraph{$\bfI\bfI$.} Using the fact that $s_h(\pi_h u - u,\phi) = 0$ since 
$\phi \in H^4(\Omega)$, the partial integration formula (\ref{eq:part-int}), 
the Cauchy-Schwarz inequality, and the interpolation error estimates we obtain
\begin{align}
A_h(\pi_h u - u,\phi) &= (\pi_h u - u, \psi)_{\Omega_h}
-(\nabla_n ( \pi_h u - u),M_{nn}(\phi))_{\partial \Omega_h}
\\
&\qquad  + (T(\pi_h u - u), \phi)_{\partial \Omega_h}
+ \gamma h^{-3}  (\pi_h u - u, \phi )_{\partial \Omega_h}
\\
 &\leq \|\pi_h u - u\|_{\Omega_h} \|\psi\|_{\Omega_h}
+\|\nabla_n ( \pi_h u - u)\|_{\partial \Omega_h}   \|M_{nn}(\phi)\|_{\partial \Omega_h}
\\
&\qquad  + \|T(\pi_h u - u) \|_{\partial \Omega_h}\| \phi\|_{\partial \Omega_h}
+ \gamma h^{-3}  \|\pi_h u - u\|_{\partial \Omega_h} \| \phi \|_{\partial \Omega_h}
\\
 &\lesssim h^4 \|u\|_{H^4(\Omega_h)} \|\psi\|_{\Omega_h}
+ h^{5/2} \|u\|_{H^4(\Omega_h)} \|M_{nn}(\phi)\|_{\partial \Omega_h}
\\
&\qquad  + h^{1/2} \|u\|_{H^4(\Omega_h)}  \| \phi\|_{\partial \Omega_h}
+ \gamma h^{-3} h^{7/2}  \|u\|_{H^4(\Omega_h)}  \| \phi \|_{\partial \Omega_h}
\\
&\lesssim 
\underbrace{(h^4 + h^{5/2} h^2 + h^{1/2} h^4)}_{\lesssim h^4} \|u\|_{H^4(\Omega_h)} \|\psi\|_{\Omega_h}
\end{align}
Here we finally used the bounds
\begin{align}
\|M_{nn}(\phi) \|_{\partial \Omega_h} 
\lesssim \delta^{1/2} \| \phi \|_{H^3(U_\delta(\partial \Omega))} 
\lesssim h^2\| \phi \|_{H^4(\Omega)}\lesssim h^2\| \psi \|_{\Omega}
\end{align}
where we used (\ref{eq:omegah-dist}) to conclude that $\partial \Omega_h \subset U_\delta(\partial \Omega)$ with $\delta \sim h^4$, and
\begin{align}
\| \phi \|_{\partial \Omega_h} 
&\lesssim \delta^{1/2} \| \phi \|_{H^1(U_\delta(\partial \Omega))} 
\lesssim \delta \| \phi \|_{W^1_\infty(U_\delta(\Omega))}
\\
&\qquad \lesssim h^4 \| \phi \|_{W^1_\infty(\Omega \cup U_\delta)}
\lesssim h^4  \| \phi \|_{H^4(\Omega)} h^4  \| \psi \|_{\Omega_h}
\end{align}

\paragraph{$\bfI\bfI\bfI$.} Using (\ref{eq:consistency-bound}) we obtain
\begin{align}
|A_h( u,\phi_h) - (f, \phi_h)_{\Omega_h}|
&\lesssim h^4 \| u \|_{H^4(\Omega)} \tn \phi_h \tn_{h,\bigstar}
\lesssim h^4 \| u \|_{H^4(\Omega)}  \| \psi \|_{\Omega_h}
\end{align}
where we used the estimate 
\begin{align}
\tn \phi_h \tn^2_{h,\bigstar} 
&\lesssim \tn \phi_h -\phi \tn^2_{h,\bigstar} + \tn \phi \tn^2_{h,\bigstar} 
\\
&\lesssim (1 + h^{-3} )\tn \phi_h -\phi \tn^2_{h}  
+ \tn \phi \tn^2_{h} + \| T(\phi) \|^2_{\partial \Omega_h} + h^{-3} \| \phi \|^2_{\partial \omega_h}  
\\
&\lesssim (1 + h^{-3})h^4 \| \phi \|^2_{H^4(\Omega)}
+ \tn \phi \tn^2_{h} + (1 + h^{-3} \delta )\| \phi \|^2_{H^4(\Omega)}
\\
&\lesssim \|\psi\|_{\Omega_h}
\end{align}
where we used (\ref{eq:wbnd-bound-strong}). 

Collecting the estimates of Terms $I-III$ completes the proof.
\end{proof}

\section{Numerics}

\subsection{Implementation} 

We consider two higher order approximations of the boundary: a piecewise cubic $C^0$ approximation or a piecewise cubic $C^1$ approximation. The steps to create the approximate boundary are as follows.
\begin{enumerate}
\item The elements cut by the boundary are located, Fig. \ref{fig:cut1}.
\item Straight segments connecting the intersection points between the boundary and the elements are established, and the geometrical object inside the domain is triangulated for ease if integration, Fig. \ref{fig:cut2}. 
\item The end points of the boundary segments and the inclinations at the endpoints (computed by use of tangent vectors) is used to obtain a $C^1$ interpolant of the boundary, Fig. \ref{fig:cut3}. (This step is skipped in the case of a $C^0$ approximation of the boundary.)
\item The geometry is approximated by a cubic triangle, interpolating the exact boundary ($C^0$ case) or the spline boundary ($C^1$ case), Fig \ref{fig:cut4}.
\end{enumerate}
Note that the approximation of the boundary may partly land outside the element. In such cases, the basis functions of the element containing the straight segment is used also outside of the element.
\subsection{Example}

We consider a circular simply supported plate under uniform load $p$. The plate is of radius $R=1/2$ and has its center at $x=1/2$, $y=1/2$.
Defining $r$ as the distance from the midpoint we then have the exact solution
\[
u=\frac{pR^4}{64\kappa}\left(1-\left(\frac{r}{R}\right)^2\right)\left(\frac{5+\nu}{1+\nu}-\left(\frac{r}{R}\right)^2\right)
\]
see, e.g., \cite{HuJo08}. The constitutive parameters were chosen as $E=10^2$, $\nu=0.3$, $t=10^{-1}$, and the stabilization parameters as
$\beta = 10^{-1}$, $\gamma = 10^2(2\kappa + 2\kappa\nu (1-\nu)^{-1})$.

We compare the convergence in normalized ($|| u- u_h||/||u||$) $L_2$, $H^1$ and piecewise $H^2$ norms in Fig. \ref{fig:conv}. These norms are computed on the discrete geometry, for simplicity the straight segment geometry. The solid lines indicate second, third, and fourth order convergence, respectively from top to bottom, and we note that we observe a slightly higher than optimal rate of convercence of about $O(h^{1/2})$ in all norms.
We note that the continuity of the approximation of the boundary seems not to be crucial as the convergence curves are very close.

Finally, in Fig. \ref{fig:elev} we show an elevation of the solution on one of the meshes in the sequence.

\section{Concluding Remarks}

We have proposed and analyzed a cut finite element method for a rectangular plate element, allowing for curved boundaries. The analysis shows that the method is optimally order convergent and stable. Two different approximations of the boundary have been tested, a standard cubic interpolation of the exact boundary and a cubic spline approximation leading to a continuously differentiable approximation of the boundary. Numerical results indicate that the continuity of the boundary approximation is not crucial. With our method, the simple rectangular element has greatly increased its practical applicability.

\newpage

\begin{figure}
\begin{center}
\includegraphics[scale=0.3]{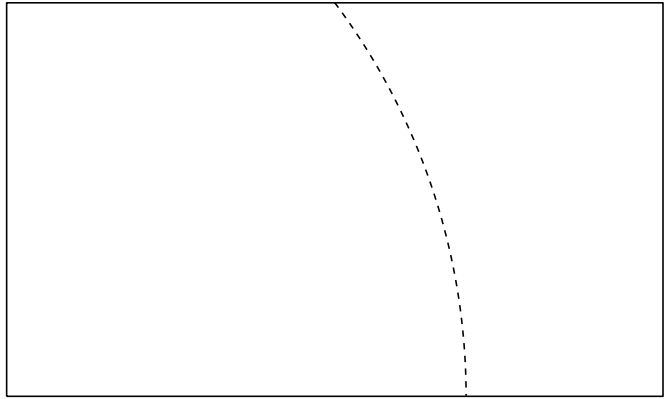}
\end{center}
\caption{Element intersected by the boundary (dashed).\label{fig:cut1}}
\end{figure}
\begin{figure}
\begin{center}
\includegraphics[scale=0.2]{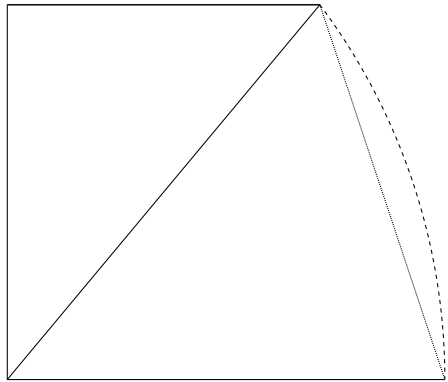}
\end{center}
\caption{Straight line approximation of the boundary (dotted) and triangulation for integration purposes.\label{fig:cut2}}
\end{figure}
\begin{figure}
\begin{center}
\includegraphics[scale=0.2]{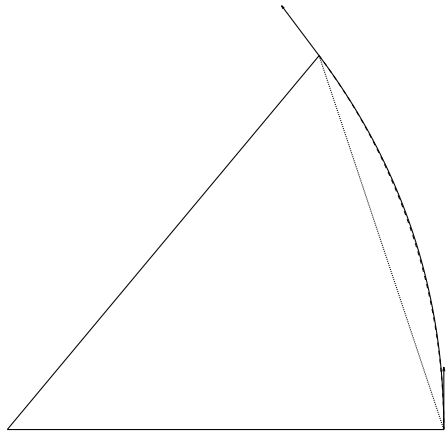}
\end{center}
\caption{Cubic spline approximation os the boundary (solid line).\label{fig:cut3}}
\end{figure}
\begin{figure}
\begin{center}
\includegraphics[scale=0.2]{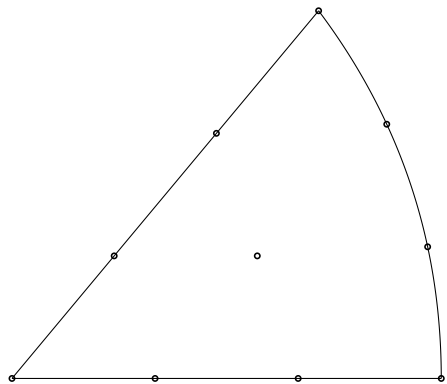}
\end{center}
\caption{Isoparametric cubic triangle approximation of the geometry.\label{fig:cut4}}
\end{figure}
\begin{figure}
\begin{center}
\includegraphics[scale=0.3]{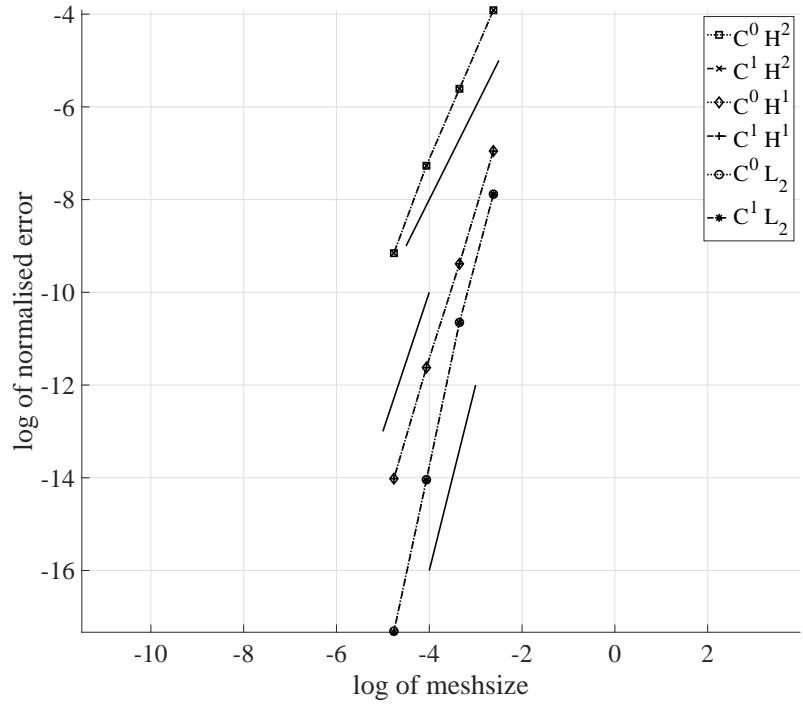}
\end{center}
\caption{Convergence in normalized $L_2$, $H^1$, and piecewise $H^2$ norms.\label{fig:conv}}
\end{figure}

\begin{figure}
\begin{center}
\includegraphics[scale=0.3]{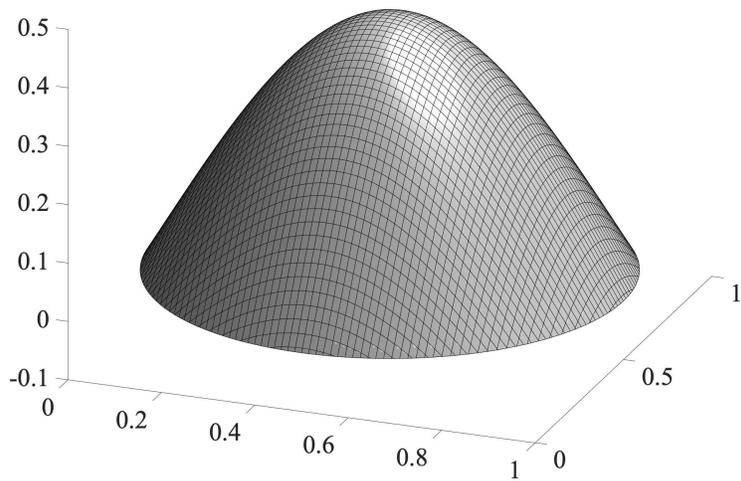}
\end{center}
\caption{Elevation of the discrete solution on one of the meshes in the sequence.\label{fig:elev}}
\end{figure}

\end{document}